\documentclass{article}
\newcommand\minput[1]{{\color{blue}\input{#1}}}
\newcommand{\cp}{\, \square \,}
\renewcommand\minput[1]{}
\usepackage[utf8]{inputenc}
\usepackage{amsthm}
\usepackage{amsfonts}
\usepackage{amsmath}
\usepackage{tikz}
\usepackage{tkz-berge}
\usepackage{boxedminipage}

\newtheorem{theorem}{Theorem}[section]
\newtheorem{conjecture}[theorem]{Conjecture}
\newtheorem{corollary}[theorem]{Corollary}

\newtheorem{lemma}[theorem]{Lemma}
\newtheorem{proposition}[theorem]{Proposition}

\newtheorem{question}{Question}

\newtheorem{remark}[theorem]{Remark}
\theoremstyle{definition}
\newtheorem{definition}[theorem]{Definition}

\begin{document}

\title{Sequence b-colorings in graphs}

\author{
Marko Jakovac $^{a,b,}$\thanks{\texttt{marko.jakovac@um.si} (corresponding author)}
\and Michael S.\ Lang $^{c,}$\thanks{\texttt{mlang@bradley.edu}} 
}

\maketitle

\begin{center}
$^a$ Faculty of Natural Sciences and Mathematics, University of Maribor, Maribor, Slovenia \\
\medskip
$^b$ Institute of Mathematics, Physics and Mechanics, Ljubljana, Slovenia\\
\medskip
$^c$ Mathematics Department, Bradley University, Peoria, Illinois, USA\\
\medskip
\end{center}

\begin{abstract}
We introduce and begin the study of sequence b-colorings, a natural generalization of the classical notion of b-colorings introduced by Irving and Manlove in 1999. In a sequence b-coloring, each color class is required to contain a prescribed minimum number of color-dominating vertices (CDVs). We establish several fundamental properties of the associated parameters, prove that every sequence is realizable, and show that the problem of deciding whether a particular graph realizes a particular sequence is NP-complete. We also characterize the sequences realized by cycles, obtain results on regular graphs with prescribed girth, and investigate colorings requiring one additional CDV, including a characterization of connected graphs with chromatic number $3$ for which no such coloring exists.
\end{abstract}

\vspace{5mm}


\vspace{5mm}

\noindent
{\bf AMS Subj.\ Class.\ (2020)}: 05C15

%-----------------
\section{Introduction}
%-----------------

The b-chromatic number was introduced by Irving and Manlove \cite{Irving} as a parameter related to proper colorings and the chromatic number. A vertex of a proper coloring is said to be color-dominating if every color used in the coloring appears in its closed neighborhood. A b-coloring is a proper coloring in which every color class contains at least one color-dominating vertex (CDV), while the b-chromatic number is the largest number of colors for which such a coloring exists. It is worth noting that the smallest number of colors in a b-coloring is precisely the chromatic number of the graph. Thus, the chromatic number and the b-chromatic number are the minimum and maximum numbers of colors in a b-coloring, respectively.

Since its introduction, the b-chromatic number has become a well-established topic in graph coloring. Already Irving and Manlove \cite{Irving} established the fundamental properties of the parameter, proved that its determination is NP-hard in general, 
and presented a polynomial-time algorithm for trees. Since then, the parameter has been studied in numerous graph classes and from several different perspectives. Among the most investigated directions are regular graphs, where the existence of b-colorings with the maximum possible number of colors has been studied extensively by several authors \cite{BlidiaRegular,Cabello,Dettlaff,SahiliRegular,ShaebaniRegNoQuad}. The parameter has also been investigated for graph products, graph powers, graph modifications, and many special graph classes. We refer the reader to the survey \cite{JakovacSurvey} and the references therein for a broader overview of the area.

In recent years, several variants of the original concept have appeared. Examples include the acyclic b-chromatic number \cite{Anholcer}, the star b-chromatic number \cite{Bozovic}, b-greedy colorings and z-colorings \cite{Costa}, and total b-chromatic colorings \cite{Mendoza}. Although these concepts arise from different coloring models, they all retain the same fundamental requirement: every color class must contain at least one CDV.

The present paper is motivated by the observation that the requirement of having one CDV in each color class is somewhat arbitrary. Once a coloring contains one CDV of a given color, it is natural to ask whether additional CDVs of the same color may also exist and, more generally, how many CDVs should be required in each color class. This leads naturally to sequences of positive integers. Instead of requiring one CDV in each color class, we prescribe how many CDVs need to appear in each color class. In this way we obtain a family of coloring parameters that extends the classical theory of b-colorings. The original b-chromatic number appears as a special case corresponding to the constant sequence consisting entirely of $1$s.

The aim of this paper is to initiate the study of these sequence b-colorings. We introduce the basic concepts and investigate their fundamental properties. We define a natural extension of the classical $m$-degree, establish corresponding upper bounds, prove that every sequence can be realized by some graph, and show that the problem of deciding whether a given graph realizes a given sequence is NP-complete. Besides the general theory, we also investigate several natural special cases. In particular, we consider colorings in which the classical b-coloring requirement is strengthened only slightly, by requiring the existence of one additional CDV beyond those required in an ordinary b-coloring. This seemingly modest modification already exhibits interesting behavior and leads to several questions that appear to be closely related to the underlying structure of b-colorings.

The paper is organized as follows. In the next section we introduce notation and definitions together with several preliminary observations. We then study realizability questions and the computational complexity of sequence b-coloring. The following section studies cycles, the simplest nontrivial class of regular graphs. Although their structure is very simple, obtaining a complete characterization of the sequences they realize is a nontrivial problem. This provides motivation and intuition for the more general results on regular graphs presented in the subsequent section. After investigating colorings requiring one additional CDV, we conclude with suggestions for further research.

%-----------------
\section{Definitions and initial results}\label{s:defs}
%-----------------

We consider only finite simple connected graphs. Let $G$ be a graph. We use $V(G)$ and $E(G)$ to denote its vertex set and edge set, respectively. The complement of $G$ is denoted by $\overline{G}$; two vertices are adjacent in $\overline{G}$ if and only if they are nonadjacent in $G$. For a vertex $x\in V(G)$, the open neighborhood of $x$ is denoted by $N(x)$ and consists of all vertices adjacent to $x$. The closed neighborhood of $x$ is denoted by $N[x]$ and is defined by $N[x]=N(x)\cup\{x\}$.
The degree of a vertex $x$ is denoted by $d_G(x)$.
The minimum and maximum degrees of $G$ are denoted by $\delta(G)$ and $\Delta(G)$, respectively. The girth of $G$, denoted by $g(G)$, is the length of a shortest cycle in $G$.
We use $K_n$, $C_n$ and $P_n$ to denote the complete graph, cycle and path, respectively, on $n$ vertices.
For graphs $G$ and $H$, the Cartesian product $G\cp H$ is the graph with vertex set $V(G\cp H)=V(G)\times V(H)$, where two vertices $(u,v)$ and $(u',v')$ are adjacent whenever either $u=u'$ and $vv'\in E(H)$, or $v=v'$ and $uu'\in E(G)$.

A proper coloring of a graph is an assignment of colors to its vertices such that adjacent vertices receive different colors. The chromatic number of a graph $G$, denoted by $\chi(G)$, is the minimum number of colors in a proper coloring of $G$.
A vertex is called a color-dominating vertex (or CDV for short) if its closed neighborhood contains all colors used in the coloring. A b-coloring is a proper coloring in which every color class contains at least one CDV. It is worth noting that any proper coloring of $G$ with $\chi(G)$ colors is necessarily a b-coloring of $G$. Indeed, if some color class contained no CDV, then each vertex of that color would miss at least one color in its closed neighborhood. Consequently, each vertex of that color class could be recolored with a color missing from its closed neighborhood, producing a proper coloring with fewer than $\chi(G)$ colors, a contradiction. Thus, the chromatic number is the minimum number of colors in a b-coloring. Finally, the b-chromatic number of $G$, denoted by $\varphi(G)$, is the maximum number of colors in a b-coloring of $G$.

Throughout the paper, a sequence means a non-increasing (finite or infinite) sequence of positive integers. Given a sequence $S$, its $i$th term is denoted by $s_i$. We use the notation $a^r$ to indicate that the term $a$ is repeated $r$ times, i.e., $(3,2^2,1^3)=(3,2,2,1,1,1)$. Similarly, $a^\infty$ denotes the infinite constant sequence whose every term equals $a$, i.e., $(a^\infty)=(a,a,a,\ldots)$. For two sequences $S$ and $T$ of the same length, we write $S\leq T$ if $s_i\leq t_i$ for every index $i$.

\begin{definition}\label{d:sequence}
Let $S=(s_1,\ldots,s_k)$, where $k\ge 1$, be a finite sequence and let $G$ be a graph.
Consider a b-coloring of $G$ with $k$ colors.
We say that this coloring {\em realizes} $S$ if the $i$th color class contains at least $s_i$ CDVs for each $i\in\{1,\ldots,k\}$.
We consider $G$ itself to realize $S$ if it admits such a b-coloring.
If $S=(s_1,s_2,\ldots)$ is an infinite sequence, then we say $G$ realizes $S$ if it realizes $(s_1,\ldots,s_k)$ for at least one integer $k$. A realization of a sequence $S$ is called {\em maximal} if the graph $G$ does not realize any sequence $T>S$.
\end{definition}

Not every graph realizes every sequence. For example, to properly color the complete graph $K_n$, one must give each vertex a different color. Since each color class consists of a single vertex, each color class contains exactly one CDV. Consequently, the only finite sequence that $K_n$ realizes is $(1^n)$.

\begin{definition}
Given an infinite sequence $S$, the set of all integers $k$ for which a graph $G$ realizes $(s_1,\ldots,s_k)$ is called the {\em $S$-spectrum} of $G$.
\end{definition}

\begin{definition}
Let $G$ be a graph and let $S$ be a sequence that is realized by $G$. We denote the minimum and maximum elements of the $S$-spectrum of $G$ by $\chi_S(G)$ and $\varphi_S(G)$, respectively.
\end{definition}

The constant sequence $(1^\infty)$ corresponds to the classical notion of b-coloring. Hence,
$$
\chi_{(1^\infty)}(G)=\chi(G) \qquad \mbox{and}\qquad \varphi_{(1^\infty)}(G)=\varphi(G).
$$
Thus, sequence b-colorings naturally extend the classical theory. Before turning to the main results, we establish several basic properties of sequence b-colorings. The following propositions, remarks and examples illustrate some immediate consequences of the definitions and provide further insight into the behavior of the newly introduced parameters.

\begin{proposition}\label{p:ineqs}
Let $G$ be a graph and suppose that $G$ realizes sequences $S$ and $T$. If $S\leq T$, then
$$\chi(G) \le \chi_S(G) \le \chi_T(G) \le \varphi_T(G) \le \varphi_S(G) \le \varphi(G) \le \Delta(G)+1.$$
\end{proposition}

\begin{proof}
The inequalities $\chi(G)\leq\chi_S(G)$, $\chi_T(G)\leq\varphi_T(G)$ and $\varphi_S(G)\leq\varphi(G)\leq\Delta(G)+1$ follow directly from the definitions.
Since $S\le T$, a coloring of $G$ that realizes $T$ also realizes $S$.
Hence $\chi_S(G)\le \chi_T(G)$ and $\varphi_T(G)\le \varphi_S(G)$, which completes the proof.
\end{proof}

\medskip

In the classical theory, Irving and Manlove \cite{Irving} introduced the parameter $m(G)$, also called the $m$-degree of a graph $G$, as a useful upper bound for the b-chromatic number. The definition of $m(G)$ is based on the number of vertices of sufficiently large degree. We now introduce a natural analogue for sequence b-colorings. For a graph $G$ and an integer $k$, let
$$n_k(G)=\bigl|\{x\in V(G) :\ d_G(x)\ge k-1\}\bigr|.$$
Since every CDV in a b-coloring with $k$ colors has degree at least $k-1$, the number of CDVs required by a sequence imposes an immediate restriction on the graph.

\begin{remark}
\label{r:ms}
If a graph $G$ realizes $(s_1,\ldots,s_k)$, then
$$\sum_{i=1}^{k}s_i\leq n_k(G).$$
\end{remark}

Motivated by Remark~\ref{r:ms}, we define the following extension of the classical $m$-degree.

\begin{definition}
Let $G$ be a graph and let $S$ be a sequence. The {\em sequential $m$-degree} of $G$ with respect to $S$ is defined by
$$m_S(G)=\max\left\{k :\ \sum_{i=1}^{k}s_i\leq n_k(G)\right\}.$$
\end{definition}

We will also refer to the parameter $m_S(G)$ simply as the $m_S$-degree of $G$. Observe that if $S=(1^\infty)$, then $m_S(G)=m(G)$. Thus, the sequential $m$-degree is a natural extension of the classical $m$-degree.

\begin{proposition}
For every graph $G$ and every sequence $S$, it follows that 
$$\varphi_S(G)\leq m_S(G).$$
\end{proposition}

\begin{proof}
Suppose, to the contrary, that $\varphi_S(G)>m_S(G)$. Then $G$ realizes $(s_1,\ldots,s_k)$ for some $k>m_S(G)$. Hence there exists a b-coloring of $G$ with $k$ colors in which the $i$th color class contains at least $s_i$ CDVs for each $i\in\{1,\ldots,k\}$. By the definition of $m_S(G)$, since $k>m_S(G)$, we have
$$\sum_{i=1}^{k}s_i>n_k(G).$$
Thus the coloring would require more than $n_k(G)$ CDVs. However, each CDV in a b-coloring with $k$ colors must have degree at least $k-1$, since it must be adjacent to at least one vertex in each of the other $k-1$ color classes. Hence all CDVs required by the sequence must belong to the set of vertices counted by $n_k(G)$. This is impossible, because there are only $n_k(G)$ such vertices.
Therefore no such $k$ exists, and consequently $\varphi_S(G)\leq m_S(G)$.
\end{proof}

\medskip

Immediately from the definition of $m_S$-degree we also get $m_T(G)\leq m_S(G)$ if $S\leq T$. However, the relationship among $\varphi_S(G)$, $m_S(G)$ and the classical b-chromatic number is somewhat subtle, as illustrated by the following examples. Let $G=K_5\cp K_2$ and let $S=(3^\infty)$. Then $\varphi_S(G)=2$, $m_S(G)=3$ and $\varphi(G)=5$. On the other hand, if $G=C_4$ and $S=(2,1^\infty)$, then $\varphi_S(G)=\varphi(G)=2$ and $m_S(G)=3$. These examples show that the parameter $m_S(G)$ does not fit naturally into a chain of inequalities involving $\varphi_S(G)$ and $\varphi(G)$.

The concepts introduced above lead to many questions, including the following.
Which sequences can be realized at all, and by which graphs?
To what extent do results from the classical theory have analogues here?
The next section begins to address some aspects of these questions by establishing general realizability results and investigating the computational complexity of sequence b-coloring.

%-----------------
\section{Realizability of sequences and complexity}
%-----------------

Obviously, every graph $G$ realizes at least one sequence: $(1^{\chi(G)})$.
A natural question is whether every sequence can be realized by at least one graph. 
We begin this section by showing that this is indeed the case. 
We introduce two simple operations that preserve realizability. 
These operations are then used in an inductive proof that every finite sequence is realizable, which shows that all sequences can be realized.

\begin{lemma}\label{l:satIter}
For some integer $k\ge1$, let $(s_1,\ldots,s_k)$ be a realizable sequence.
Then the following statements hold.
\begin{enumerate}
    \item $(s_1,\ldots,s_{k-1},s_k,1)$ is realizable. \label{i:,1}
    \item $(s_1,\ldots,s_{k-1},s_k+1)$ is realizable. (Assume $s_k<s_{k-1}$ if $k>1$ to maintain order.) \label{i:+1}
\end{enumerate}
\end{lemma}

\begin{proof}
Let $G$ be a graph with a coloring realizing $(s_1,\ldots,s_k)$.

To prove \ref{i:,1}, we construct a graph $G'$ from $G$ by adding a new vertex $u$ and joining it to each of the CDVs from each of the $k$ color classes of $G$. Assign $u$ a new color $k+1$. Since no other vertex has color $k+1$, the coloring of $G'$ is proper. Moreover, $u$ is adjacent to a vertex of every color used in $G$, and its closed neighborhood also contains its own color. Hence $u$ is a CDV of color $k+1$. Each previously existing CDV remains a CDV, since it is adjacent to $u$. Thus $G'$ realizes $(s_1,\ldots,s_{k-1},s_k,1)$.

To prove \ref{i:+1}, we construct a graph $G''$ from $G$ by adding a new vertex $v$ and joining it to one CDV from each color class except the $k$th one. Assign color $k$ to vertex $v$. Since $v$ has no neighbor of color $k$, the resulting coloring is proper. Since $N[v]$ contains color $k$ through $v$ itself and every other color through $v$'s neighbors, $v$ is an additional CDV of color $k$. Furthermore, every previously existing CDV remains a CDV, since no new color has been introduced. Therefore $G''$ realizes $(s_1,\ldots,s_{k-1},s_k+1)$.
\end{proof}

\medskip

\begin{theorem}\label{t:finSeqSat}
Every sequence is realizable.
\end{theorem}

\begin{proof}
Let $S=(s_1,\ldots,s_k)$ be an arbitrary finite sequence.
We proceed by induction on $k$.

Suppose that $k=1$, so $S=(s_1)$. The graph $K_1$ realizes the trivial sequence $(1)$. By applying part~\ref{i:+1} of Lemma~\ref{l:satIter} 
exactly $s_1-1$ times, we find that $(s_1)=S$ is realizable.

Now suppose that $k>1$ and that every sequence of length $k-1$ is realizable. Recall that $S=(s_1,\ldots,s_k)$. By the induction hypothesis, the sequence $(s_1,\ldots,s_{k-1})$ is realizable. Applying part~\ref{i:,1} of Lemma~\ref{l:satIter}, we see that $(s_1,\ldots,s_{k-1},1)$ is realizable. Now, applying part~\ref{i:+1} of Lemma~\ref{l:satIter} exactly $s_k-1$ times 
shows that $(s_1,\ldots,s_{k-1},s_k)=S$ is realizable.

Finally, consider an arbitrary infinite sequence $S'=(s_1,s_2,\ldots)$.
Since we have already shown that the initial subsequence $(s_1)$ is realizable, so is $S'$.
\end{proof}

\medskip

Theorem~\ref{t:finSeqSat} 
completely resolves
the realizability problem.
Consequently, attention shifts from the existence of realizations to their structure. 
Given a fixed sequence, one may ask which graphs realize it. 
Conversely, for a fixed graph, one may ask which sequences it realizes.
Combining the two questions asks whether a given graph realizes a given sequence.

We conclude this section by investigating this third problem from the computational point of view. 
Unsurprisingly, determining $\varphi_S(G)$ for particular $S$ and $G$ is difficult in general. We show that this problem is NP-hard, even when the sequence finishes with 1s. Since we only deal with finite graphs, this assumption is not very restrictive.

Formally, we consider the following decision problem.

\begin{boxedminipage}{0.85\textwidth}
{\sc Name}: SEQUENCE-B-COLORING\\
{\sc Instance:} A sequence $S$ for which there is an integer $q$ such that $s_i=1$ for all $i>q$,\\
\hspace*{.71in}a graph $G$,\\
\hspace*{.71in}and an integer $k$\\
{\sc Question:} Is $\varphi_S(G)\geq k$?\\
\hspace*{.73in}That is, does $G$ realize $(s_1,\ldots,s_r)$ for some $r\geq k$?
\end{boxedminipage}

\begin{theorem}
The problem SEQUENCE-B-COLORING is NP-complete.
\end{theorem}

\begin{proof}
It is easy to see that SEQUENCE-B-COLORING belongs to NP. Indeed, given a coloring of $G$, we can check in polynomial time that it is proper, determine all CDVs, and verify that every color class contains at least the required number of CDVs. To prove NP-hardness, we give a polynomial-time reduction from B-CHROMATIC NUMBER, which was shown in \cite[Theorem~8]{Irving} to be NP-complete.

Let $G$ be a graph and let $k$ be a positive integer. Since $S$ ends with 1s, there is an integer $q$ such that $s_i=1$ for every $i>q$. We construct a graph $H$ from $G$ as follows. For every $i\in\{1,\ldots,q\}$, add an independent set $U_i$ of $s_i$ new vertices. Join every vertex in $U_i$ to every vertex in $U_j$ whenever $i\neq j$. Finally, join every new vertex to every vertex of $G$. The graph $H$ can be constructed in polynomial time. We prove that $\varphi(G)\geq k$ if and only if $\varphi_S(H)\geq q+k$.

Assume first that $\varphi(G)\geq k$. Then $G$ has a b-coloring using $r$ colors for some $r\geq k$. Rename its colors as $q+1,\ldots,q+r$. Extend this coloring to $H$ by assigning color $i$ to every vertex of $U_i$, for every $1\leq i\leq q$. The obtained coloring is proper. Each set $U_i$ is independent, vertices from different sets $U_i$ and $U_j$ are adjacent, and every new vertex is adjacent to every vertex of $G$. Let $u\in U_i$. The closed neighborhood of $u$ contains color $i$ through $u$ itself, every color $j\in\{1,\ldots,q\}\setminus\{i\}$ through the vertices of $U_j$, and every color in $\{q+1,\ldots,q+r\}$ through the vertices of $G$. Since every vertex of $U_i$ has the same neighborhood, every vertex of $U_i$ is a CDV. Therefore color $i$ contains exactly $s_i$ CDVs for every $1\leq i\leq q$. 
Each CDV of the original coloring of $G$ remains a CDV in $H$, since it is adjacent to all of the new vertices and therefore its closed neighborhood contains all colors $1,\ldots,q$.
Thus the extended coloring is a b-coloring realizing $(s_1,\ldots,s_{q+r})$. Hence $\varphi_S(H)\geq q+r\geq q+k$.

Now assume that $\varphi_S(H)\geq q+k$. Then $H$ has a b-coloring realizing $(s_1,\ldots,s_r)$ for some $r\geq q+k$. First, consider the colors of the new vertices. Since every vertex of $U_i$ is adjacent to every vertex outside $U_i$, a color used in $U_i$ cannot appear outside $U_i$. We also show that every set $U_i$ is monochromatic. Suppose that two different colors appear in $U_i$. Since $U_i$ is independent, no vertex of one of these colors is adjacent to a vertex of the other color. Therefore no vertex of either color in $U_i$ can be a CDV. This is impossible, because every color used in the coloring must contain at least one CDV. Thus every set $U_i$ receives exactly one color. Since vertices from different sets $U_i$ and $U_j$ are adjacent, these $q$ colors are pairwise different. Also, none of these colors appears in $G$. Therefore exactly $q$ colors are used on the new vertices, while the remaining $r-q$ colors are used on $G$. Now consider the restriction of the coloring to $G$. It is a proper coloring of $G$ with $r-q$ colors. Let $c$ be one of these colors. Since $c$ does not appear on the new vertices and the coloring of $H$ is a b-coloring, color $c$ has a CDV in $G$. This vertex remains a CDV after restricting the coloring to $G$, because its closed neighborhood still contains every color used on $G$. Thus every color used on $G$ still contains a CDV. Therefore the restriction is a b-coloring of $G$ using $r-q$ colors, and hence $\varphi(G)\geq r-q\geq k$.

We have proved that $\varphi(G)\geq k$ if and only if $\varphi_S(H)\geq q+k$. Therefore SEQUENCE-B-COLORING is NP-hard, and consequently it is NP-complete.
\end{proof}

%-----------------
\section{Cycles}\label{s:cycles}
%-----------------

Cycles form the simplest nontrivial class of connected regular graphs. Although their structure is elementary, determining all sequences that they realize is already a rather involved problem. Their study therefore provides a useful step toward the more general results on regular graphs in the next section.

Let us recall some relevant facts.
For $n$ even, $\chi(C_n)=2$.
Moreover, the cycle on four vertices has $\varphi(C_4)=2$, while each even cycle on at least six vertices has $\varphi(C_n)=3$.
Odd cycles have $\chi(C_n)=\varphi(C_n)=3$.

We observe that sequences with two terms cannot be realized by odd cycles.
Also, the only way to properly color the even cycle $C_{2r}$ with two colors is to alternate those colors, yielding exactly $r$ CDVs in each color class.
Hence, $C_{2r}$ realizes the sequence $(r,r)$ but does not realize $(r+1,1)$. 

We now turn to sequence b-colorings with three colors.

\begin{lemma}\label{l:ab,2a+b}
If the cycle $C_n$ admits a b-coloring with three colors, with at least $r$ vertices of color $1$ and at least $s$ CDVs of color $2$, then $n\ge 2r+s$.
\end{lemma}

\begin{proof}
Let $X$ be the set of vertices colored $1$, and write $|X|=r'\ge r$. The vertices of $X$ determine $r'$ internally vertex-disjoint paths whose endpoints are consecutive vertices of $X$ along the cycle. Since the coloring is proper, each of these paths contains at least one internal vertex.

If such a path contains exactly one internal vertex, then this vertex is not a CDV. If it contains at least two internal vertices, then propriety requires that their colors alternate between $2$ and $3$, and exactly the ones on the ends are CDVs. Consequently, the path contains at most two CDVs of color 2.
To contain a single CDV of color 2, the path must have at least one extra (beyond the required one) internal vertex.
To contain two CDVs of color 2, the path must have at least two extra internal vertices.

For the full cycle to contain at least $s$ CDVs of color 2, then, the $r'$ paths together need to have at least $s$ extra internal vertices.
The graph thus has at least the $r'$ vertices in $X$, the $r'$ required internal vertices and at least $s$ extra internal vertices.
Now $n\geq r'+r'+s\geq 2r+s$, as indicated.
\end{proof}

\medskip

\begin{corollary}\label{c:abc,2a+b}
If the cycle $C_n$ realizes the sequence $(r,s,t)$, then $n\geq 2r+s$.
\end{corollary}
\begin{proof}
A coloring realizing $(r,s,t)$ is a b-coloring with three colors containing at least $r$ vertices of one color and at least $s$ CDVs of another color. The result therefore follows directly from Lemma~\ref{l:ab,2a+b}.
\end{proof}

\medskip

Notice that the parameter $t$ does not appear in the bound. This is because CDVs of colors $2$ and $3$ may lie on the same path between two consecutive vertices of color $1$. The next lemma shows that the parity of these paths yields a stronger necessary condition in certain cases.

\begin{lemma}\label{l:n-a,even}
If the cycle $C_n$ admits a b-coloring with three colors in which exactly $r$ vertices are colored $1$, and all of them are CDVs, then $n-r$ is even.
\end{lemma}

\begin{proof}
Let $X$ be the set of vertices colored $1$. As in the proof of Lemma~\ref{l:ab,2a+b}, the vertices of $X$ determine $r$ internally vertex-disjoint paths whose endpoints are consecutive vertices of $X$ along the cycle. Since the coloring is proper, the internal vertices of each such path alternate between colors $2$ and $3$, which we consider opposites.

Since every vertex of $X$ is a CDV, its two neighbors, which are the last internal vertex of one path and the first internal vertex of the next, have opposite colors.
Consider one of these paths, and let $m$ be its number of internal vertices. If $m$ is odd, then the first and last internal vertices of this path have the same color. If $m$ is even, then they have opposite colors.
Therefore, the first internal vertex of this path and the first internal vertex of the next path have opposite colors exactly when $m$ is odd.

Starting at the initial internal vertex of some path and traversing the entire cycle, we return to this starting vertex with its given color. Hence, the number of paths containing an odd number of internal vertices must be even.
This number has the same parity as the total number of internal vertices over all the paths.
This total is $n-r$, so $n-r$ is even.
\end{proof}

\medskip

\begin{corollary}\label{c:n-aodd,2a+b+2}
If the cycle $C_n$ realizes the sequence $(r,s,t)$ and $n-r$ is odd, then $n\ge 2r+s+2$.
\end{corollary}

\begin{proof}
Suppose, to the contrary, that $C_n$ realizes $(r,s,t)$ and $n-r$ is odd but
$n<2r+s+2$. Let color $1$ be the color whose class contains at least $r$ CDVs, and let $r'\geq r$ be the total number of vertices colored $1$. Similarly, let color $2$ be a color whose class contains at least $s$ CDVs. Applying Lemma~\ref{l:ab,2a+b} to colors $1$ and $2$ gives $n\ge 2r'+s$. If $r'\ge r+1$, then $n\ge 2(r+1)+s=2r+s+2$, contrary to our assumption. Therefore, $r'=r$. Since the color class contains at least $r$ CDVs and exactly $r$ vertices, every vertex colored $1$ is a CDV. Lemma~\ref{l:n-a,even} now implies that $n-r$ is even, contrary to the assumption that $n-r$ is odd. Therefore, $n\ge 2r+s+2$, as required.
\end{proof}

\medskip

We now turn to the corresponding constructions. The following two lemmas show that the bounds obtained above are tight in the cases relevant for the complete characterization.

\begin{lemma}\label{l:4r+3s,2r+sss}
If $n=4r+3s$ for integers $r\ge 0$ and $s\ge 1$, then the cycle $C_n$ realizes $(2r+s,s,s)$.
\end{lemma}
\begin{proof}
Going around the cycle, first color $4r$ consecutive vertices by repeating the pattern $1,2,1,3$ exactly $r$ times, and then color the remaining $3s$ consecutive vertices by repeating the pattern $1,2,3$
exactly $s$ times. This defines a proper coloring of $C_n$.

Each occurrence of the pattern $1,2,1,3$ contributes two CDVs of color $1$. Similarly, each occurrence of the pattern $1,2,3$ contributes one CDV of each color. The same remains true at the junctions between consecutive occurrences of the patterns and between the last and first vertices of the cycle. Therefore, the resulting coloring contains exactly $2r+s$ CDVs of color $1$, and exactly $s$ CDVs of each of colors $2$ and $3$. Hence it realizes the sequence $(2r+s,s,s)$.
\end{proof}

\medskip

\begin{lemma}\label{l:3k+2,kkk}
If $n=3s+2$ for some integer $s\geq1$, then the cycle $C_n$ realizes the sequence $(s,s,s)$.
\end{lemma}
\begin{proof}
Going around the cycle, color the first $3s$ consecutive vertices by repeating the pattern $1,2,3$ exactly $s$ times, and color the remaining two vertices with colors $1$ and $3$, respectively. This defines a proper coloring of $C_n$.

Every vertex among the first $3s$ vertices has neighbors of the other two colors and is therefore a CDV. The remaining vertex of color $1$ has two neighbors of color $3$, while the remaining vertex of color $3$ has two neighbors of color $1$, so neither of them is a CDV. Consequently, each color has exactly $s$ CDVs. Hence the coloring realizes the sequence $(s,s,s)$.
\end{proof}

\medskip

The results of this section so far provide both necessary conditions and explicit constructions for three-term sequences realized by cycles. Since the constructions are based on patterns of lengths $3$ and $4$, it is natural to distinguish residue classes modulo $12$. Combining the necessary conditions with these constructions yields the following description of the three-term sequences realized by cycles.

\begin{proposition}\label{p:mod12}
Regarding three-term sequences, the cycle $C_3$ only realizes $(1,1,1)$ and $C_4$ does not realize any such sequence.
Suppose $n\geq5$ and write $n=12k+j$ for $k\geq0$ and $0\leq j<12$.
The first table below lists maximal three-term sequences realized by the cycle $C_n$,
while the second table lists minimal three-term sequences not realized by $C_n$.

\bigskip

\begin{tabular}{r|lcl}
\hline
For $j=$&\multicolumn{3}{c}{the cycle $C_{12k+j}$ realizes}\\
\hline
0&$(6k-2i+0,4i+0,4i+0)$&for $1\leq i\leq k$&\\
1&$(6k-2i+1,4i-1,4i-1)$&for $1\leq i\leq k$&\\
2&$(6k-2i+2,4i-2,4i-2)$&for $1\leq i\leq k$&and $(4k,4k,4k)$\\
3&$(6k-2i+1,4i+1,4i+1)$&for $0\leq i\leq k$&\\
4&$(6k-2i+2,4i+0,4i+0)$&for $1\leq i\leq k$&\\
5&$(6k-2i+3,4i-1,4i-1)$&for $1\leq i\leq k$&and $(4k+1,4k+1,4k+1)$\\
6&$(6k-2i+2,4i+2,4i+2)$&for $0\leq i\leq k$&\\
7&$(6k-2i+3,4i+1,4i+1)$&for $0\leq i\leq k$&\\
8&$(6k-2i+4,4i+0,4i+0)$&for $1\leq i\leq k$&and $(4k+2,4k+2,4k+2)$\\
9&$(6k-2i+3,4i+3,4i+3)$&for $0\leq i\leq k$&\\
10&$(6k-2i+4,4i+2,4i+2)$&for $0\leq i\leq k$&\\
11&$(6k-2i+5,4i+1,4i+1)$&for $0\leq i\leq k$&and $(4k+3,4k+3,4k+3)$\\
\hline
\end{tabular}

\bigskip

\begin{tabular}{r|ll}
\hline
For $j=$&\multicolumn{2}{c}{the cycle $C_{12k+j}$ does not realize}\\
\hline
0&$(6k-2i+1,\max\{4i-3,1\},1)$&for $1\leq i\leq k$\\
1&$(6k-2i+0,\max\{4i+0,1\},1)$&for $0\leq i\leq k$\\
2&$(6k-2i+1,\max\{4i-1,1\},1)$&for $0\leq i\leq k$\\
3&$(6k-2i+2,\max\{4i-2,1\},1)$&for $0\leq i\leq k$\\
4&$(6k-2i+1,\max\{4i+1,1\},1)$&for $0\leq i\leq k$\\
5&$(6k-2i+2,\max\{4i+0,1\},1)$&for $0\leq i\leq k$\\
6&$(6k-2i+3,\max\{4i-1,1\},1)$ &for $0\leq i\leq k$\\
7&$(6k-2i+2,\max\{4i+2,1\},1)$&for $-1\leq i\leq k$\\
8&$(6k-2i+3,\max\{4i+1,1\},1)$&for $0\leq i\leq k$\\
9&$(6k-2i+4,\max\{4i+0,1\},1)$&for $0\leq i\leq k$\\
10&$(6k-2i+3,\max\{4i+3,1\},1)$&for $-1\leq i\leq k$\\
11&$(6k-2i+4,\max\{4i+2,1\},1)$&for $-1\leq i\leq k$\\
\hline
\end{tabular}

\bigskip

\noindent Together, the two tables determine all
three-term sequences realized by $C_n$.
\end{proposition}

\begin{proof}
The statements about $C_3$ and $C_4$ are clear.

Most of the realization table follows from Lemma~\ref{l:4r+3s,2r+sss}.
For each row, choose integers $r\ge 0$ and $s\ge 1$ such that $12k+j=4r+3s$.
Then Lemma~\ref{l:4r+3s,2r+sss} yields the sequence $(2r+s,s,s)$, which gives the corresponding entry in the first table. 
The $i$ intervals keep $2r+s\geq s\geq1$.
For $j\in\{2,5,8,11\}$, the additional sequence follows from Lemma~\ref{l:3k+2,kkk}, since in these cases $12k+j\equiv2\pmod 3$.

The non-realization table follows from Corollary~\ref{c:n-aodd,2a+b+2}.
For every sequence $(a,b,1)$ that is listed, we have $(12k+j)-a\equiv1\pmod 2$ and $12k+j<2a+b+2$. Hence $C_{12k+j}$ does not realize $(a,b,1)$.
The $i$ intervals keep $a\geq b$ without redundant (less than already listed) sequences.

Finally, 
Definition~\ref{d:sequence}
implies that if $S\leq T$, then a graph realizing $T$ also realizes $S$. Consequently, every sequence less than a sequence listed in the first table is also realized, while no sequence greater than a sequence listed in the second table is realized. 
One can verify that these leave no gaps.
Therefore, the sequences in the first table are maximal, those in the second are minimal, and the two tables together determine all
three-term sequences realized by $C_n$.
\end{proof}

The following two corollaries highlight two particularly interesting families of sequences. Their proofs amount to reading the corresponding extremal values from Proposition~\ref{p:mod12}.

\begin{corollary}
Suppose $n\geq5$ and write $n-2=4a+2b+c$, where $a$ is a nonnegative integer and $b,c\in\{0,1\}$. Then the maximum integer $t$ for which the cycle $C_n$ realizes $(t,1,1)$ is $t=2a+c$.
\end{corollary}

\begin{corollary}
Suppose $n\geq5$ and write $n-1=3a+b$, where $a$ is a nonnegative integer and $b\in\{0,1,2\}$. Then the maximum integer $t$ for which the cycle $C_n$ realizes $(t,t,t)$ is $t=a+b-1$.
\end{corollary}

%-----------------
\section{Regular graphs with prescribed girth}\label{s:reg}
%-----------------

A result of Kouider~\cite{Koui-04}
states that every $d$-regular graph of girth at least $6$ has b-chromatic number $d+1$. Since every $d$-regular graph $G$ has $\Delta(G)=d$ and $\varphi_S(G)\leq\Delta(G)+1$, this is also the largest possible value of $\varphi_S(G)$. Therefore, if a $d$-regular graph realizes any sequence $S$ of length $d+1$, then necessarily $\varphi_S(G)=d+1$. In particular, Kouider's theorem says that every $d$-regular graph of girth at least $6$ realizes the sequence $(1^{d+1})$. The proof starts with a vertex $x$, which is made into a CDV by assigning different colors to its neighbors. The neighbors of $x$ are then made into CDVs by appropriately coloring their remaining neighbors. The girth assumption guarantees that this construction can be completed without creating conflicts. Thus one obtains at least one CDV in each of the $d+1$ color classes.

This construction naturally raises the following question. Can the vertices in the second neighborhood of $x$ also be made into CDVs? If so, then each of the $d$ colors used on the neighbors of $x$ would have $d$ CDVs, while the color of $x$ would still have one CDV. In other words, the graph would realize the sequence $(d^d,1)$.

In this section, we investigate how far the sequence $(1^{d+1})$ can be increased under girth assumptions. We first show that one additional CDV can always be guaranteed, then study the limitations of this approach by presenting counterexamples, and finally prove that the sequence $(d^d,1)$ is realized by every $d$-regular graph of girth at least $8$.

We begin with the smallest strengthening of Kouider's theorem. It turns out that one additional CDV can always be guaranteed. In other words, we first show that the sequence $(2,1^d)$ is realized by every $d$-regular graph of girth at least $6$.

\begin{theorem}\label{t:girth6}
Let $d\ge 2$ and let $G$ be a $d$-regular graph of girth $g(G)\ge 6$.
Then $G$ realizes the sequence $(2,1^d)$.
\end{theorem}

\begin{proof}

If $d=2$, then $G$ is a cycle of length at least $6$, and the result follows from Proposition~\ref{p:mod12}. We may therefore suppose that $d\geq 3$.

Choose a vertex $x\in V(G)$. Let $N(x)=\{x_1,\ldots,x_d\}$. Assign color $d+1$ to vertex $x$ and color vertex $x_i$ with color $i$ for every $i\in\{1,\ldots,d\}$. Thus $x$ is a CDV. For every $i\in\{1,\ldots,d\}$, let $B_i=N(x_i)\setminus\{x\}$. 
Since $G$ has no cycles of length $3$,
no $x_i$ is in any $B_j$ and each $B_i$ forms an independent set.
Because there are no cycles of length 4, the sets $B_1,\ldots,B_d$ are pairwise disjoint.
Moreover, there are no edges between $B_i$ and $B_j$ for $i\neq j$, since such an edge would give a cycle of length $5$. Each set $B_i$ contains $d-1$ vertices. Color the vertices of $B_i$ with the colors in
$\{1,\ldots,d\}\setminus\{i\}$, assigning a different color to each vertex. It follows that each $N[x_i]$ contains every color: $x_i$ itself has color $i$, $x$ has color $d+1$, and all remaining colors appear on the vertices of $B_i$. Hence every vertex $x_i$ is a CDV. 

Choose an index $i\neq 1$ and let $y$ be the vertex of $B_i$ colored with color $1$. We shall modify the coloring so that $y$ also becomes a CDV. Let $A=N(y)\setminus\{x_i\}$. The $d-1$ vertices of $A$ are still uncolored. Color the vertices of $A$ with the colors in $\{1,\ldots,d+1\}\setminus\{1,i\}$, assigning a different color to each vertex. If this coloring creates no conflict with the already colored vertices, then $N[y]$ contains every color and $y$ is a CDV. We now show that all possible conflicts can be removed without changing the set of colors appearing in any $B_j$. First observe that there are no edges between $A$ and $B_i\setminus\{y\}$, since such an edge would give a $4$-cycle. Fix $j\neq i$ and consider the edges between $A$ and $B_j$. These form a matching, since if a vertex of $A$ had two neighbors in $B_j$, or if a vertex of $B_j$ had two neighbors in $A$, then $G$ would contain a $4$-cycle. The vertices of $A$ have pairwise different colors. Therefore, for every vertex of $B_j$, at most one color is forbidden by its possible neighbor in $A$, and the forbidden colors at different vertices of $B_j$ are different. We may now permute the colors inside $B_j$ so that no vertex receives the color of its neighbor in $A$. To see this, consider only the vertices of $B_j$ whose forbidden color belongs to $\{1,\ldots,d\}\setminus\{j\}$. If there are at least two such vertices, permute their forbidden colors cyclically so that no vertex receives its own forbidden color. If there is exactly one such vertex, exchange its forbidden color with the color of any other vertex of $B_j$, which exists because $d\geq 3$. The remaining colors may be assigned arbitrarily. We can perform this recoloring independently for every $j\neq i$, since the sets $B_j$ are pairwise nonadjacent and these changes do not create new conflicts. Also, every $B_j$ still contains exactly the colors $\{1,\ldots,d\}\setminus\{j\}$. Hence all vertices $x_1,\ldots,x_d$ remain CDVs.

After these recolorings, the coloring is proper on all colored vertices. Since $N[y]$ contains color $1$ on $y$, color $i$ on $x_i$, and every other color on the vertices of $A$, $y$ is a CDV. Since $x_1$ is also a CDV of color $1$, color $1$ contains at least two CDVs, while every other color contains at least one CDV. Finally, color all remaining vertices greedily. At each step, a vertex has at most $d$ colored neighbors, while $d+1$ colors are available. Therefore the coloring can be completed to a proper $(d+1)$-coloring of $G$. The CDVs constructed above remain CDVs, and hence $G$ realizes $(2,1^d)$.
\end{proof}

\medskip

Theorem~\ref{t:girth6} shows that one additional CDV can always be obtained under the same girth assumption as in Kouider's theorem. It is therefore natural to ask whether two additional CDVs of the same color can also be guaranteed. The following proposition shows that this is not possible in general. A candidate for a counterexample is the Heawood graph, which is the cubic graph of girth $6$ with the fewest vertices.
Since it satisfies the assumptions of Theorem~\ref{t:girth6}, it is natural to ask whether it realizes the sequence $(3,1^3)$ in addition to $(2,1^3)$.

\begin{proposition}\label{p:Heawood}
The Heawood graph is $3$-regular with girth 6 but does not realize the sequence $(3,1,1,1)$. 
\end{proposition}

\begin{proof}
Let $H$ be the graph whose vertices are the residue classes modulo 14,
where each even vertex $i$ is adjacent to the vertices $i-1$, $i+1$ and $i+5$.
This is a representation of the Heawood graph, drawn on the left in Figure~\ref{f:Heawood}.
We make three observations:
$H$ is 3-regular and has girth 6, as claimed;
$H$ is bipartite with diameter 3;
and regarding automorphisms, $H$ is vertex-transitive.

\begin{figure}[ht]
\centering
\begin{tikzpicture}[scale=.75]
\GraphInit[vstyle=Normal]
\grHeawood[prefix=]
\end{tikzpicture}
\hspace{1cm}
\begin{tikzpicture}[scale=.75]
\SetVertexNoLabel
\grHeawood
\AssignVertexLabel{a}{[4],1-4,23,3,34,1-2,12,2,24,4,14,1,13,1-3}
\end{tikzpicture}
\caption{Two views of the Heawood graph}
\label{f:Heawood}
\end{figure}
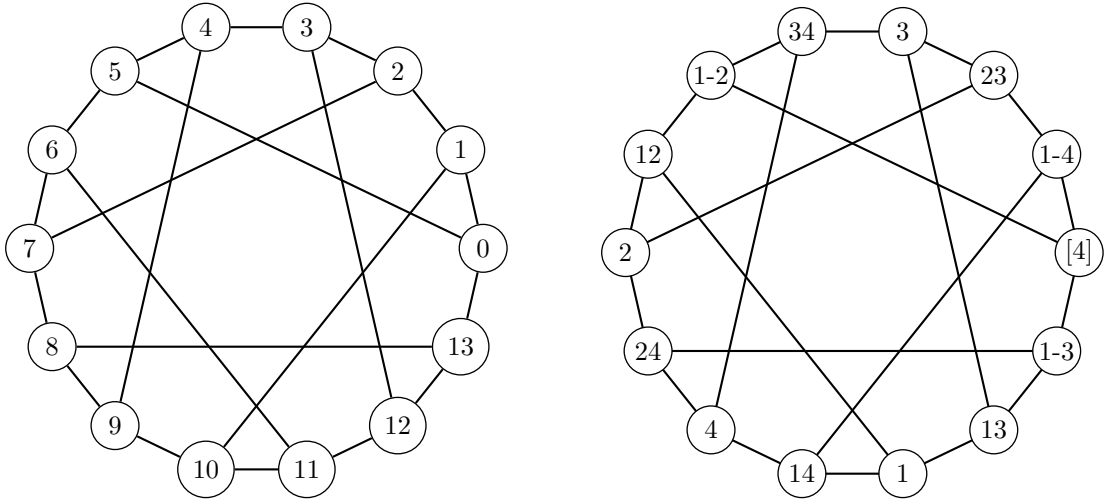

Another representation is to have the vertices be the 1- and 2-element subsets of $[4]=\{1,2,3,4\}$, together with the set itself and its partitions into two 2-element subsets (which we designate based on which element pairs with 1), with edges given by inclusion.
An example showing the equivalence is on the right of Figure~\ref{f:Heawood}.
The importance of this version is to make it clear how, fixing a vertex (the full set $[4]$, which we identify with vertex 0), we can freely permute those at distance 3 from it (the four singletons, which are then identified with vertices 3, 7, 9 and 11).

Suppose, contrary to our assertion, that $H$ realizes $(3,1,1,1)$.
Then $H$ has a proper $4$-coloring in which one color contains at least three CDVs.
We may assume that this is color $1$,
and locate three CDVs of this color. 
Combining the propriety of the coloring with the bipartiteness and diameter of the graph,
there are three possible arrangements of these three vertices:
\begin{enumerate}
\item\label{i:comNei} they have a common neighbor;
\item\label{i:altHex} they are alternating vertices of a $6$-cycle;
\item\label{i:iso} one of them is at distance $3$ from each of the other two, which are at distance $2$ from each other.
\end{enumerate}
Invoking automorphisms of $H$ and permutations of colors 2, 3 and 4,
we show in each case that there cannot be CDVs of all three of these remaining colors.

In case~\ref{i:comNei},
by moving their common neighbor to vertex $0$,
we may assume that the CDVs of color $1$ are vertices $1$, $5$ and $13$.
Observe that the neighbors of these CDVs include all of the even vertices.
By permuting colors $2$, $3$ and $4$, we may assume that vertex $0$ has color $2$, which means that the rest of the even vertices must receive colors 3 and 4.
Now, the neighbors of vertex 0 all have color 1 and the neighbors of vertices 3, 7, 9 and 11 only have neighbors colored 3 and 4.
Thus, color 2 has no CDV, which is a contradiction.

In case~\ref{i:altHex},
we choose one of the CDVs of color 1 and move it to vertex 0.
Leaving that fixed, we move the vertex opposite it in the 6-cycle to vertex 3.
Fixing both of those, we permute vertices 7, 9 and 11 so that the one that does not have a neighbor in the 6-cycle is vertex 11.
We now have that the three CDVs of color $1$ are vertices $0$, $2$, and $4$, which are alternating vertices of the $6$-cycle $0,1,2,3,4,5$. 
By permuting colors $2$, $3$ and $4$, we may assume that they are assigned to vertices $1$, $3$ and $5$, respectively.
Since vertex $0$ is a CDV, its third neighbor $13$ must receive color $3$. 
Similarly, since vertices $2$ and $4$ are CDVs, their third neighbors $7$ and $9$ must receive colors $4$ and $2$, respectively. 
Vertex $8$ is adjacent to vertices 7, 9 and 13, which have colors 4, 2 and 3, so
propriety dictates that vertex $8$ gets color $1$.
Now vertices 6, 10 and 12 each have repeated colors in their neighborhoods, so vertex $11$ is the only one that can possibly be a CDV of color $2$, $3$ or $4$. 
Therefore these three colors cannot all contain CDVs, again giving a contradiction.

In case~\ref{i:iso},
we move the CDV of color 1 that is distance 3 from the other two to vertex 0.
Leaving that fixed, we move the other two CDVs of color 1 to vertices 3 and 7.
Observe that fixing these three vertices also fixes vertices 1, 2 and 10.
By permuting colors $2$ and $3$, we may assume that they are assigned to vertices 1 and 2, respectively.
Since vertex $0$ is a CDV of color 1 and its neighbor 1 has color 2, its other neighbors $5$ and $13$ must receive colors $3$ and $4$.
Fixing vertices 0, 3 and 7, we permute vertices 9 and 11 so that vertex 5 has color 3, leaving vertex 13 with color 4.
Since vertex $3$ is a CDV, its remaining neighbors 4 and 12 must receive colors 2 and 4, with vertex 12 being the one to get color 2 for propriety.
Examining vertex $7$ similarly, we see that vertices $6$ and $8$ must receive colors $4$ and $2$, respectively.
The vertices $9$ and $11$ are adjacent to vertices of colors $2$ and $4$, so each of them must receive color $1$ or $3$.
Now vertex $10$ is the only possible CDV of color $4$.
However, with vertices colored as described so far, there is no vertex that can be a CDV of color $3$.
Thus colors $2$, $3$ and $4$ cannot all contain CDVs, yielding a final contradiction.

Each of the three possible arrangements of three CDVs of color $1$ leads to a contradiction. Therefore the Heawood graph does not realize $(3,1,1,1)$.
\end{proof}

\medskip

The Heawood graph shows that girth at least $6$ is not sufficient to guarantee two extra CDVs of the same color, since it does not realize the sequence $(3,1^3)$.
One might hope that having the two extra CDVs be of different colors might change matters, but Corollary~\ref{c:n-aodd,2a+b+2} shows that the cycle $C_7$ does not realize the sequence $(2,2,1)$.
We wonder what can be guaranteed if we combine features of these two exceptions.
The next theorem shows that if we require degree at least 3 (as in the Heawood graph but not the 7-cycle) and girth at least 7 (as in the 7-cycle but not the Heawood graph), then we can put an extra CDV into every color class.

\begin{theorem}\label{t:girth7}
Let $d\ge 3$ and let $G$ be a $d$-regular graph of girth $g(G) \ge 7$.
Then $G$ realizes the sequence $(2^{d+1})$.
\end{theorem}

\begin{proof}
We construct a proper $(d+1)$-coloring of $G$ with at least two CDVs of each color.

Choose adjacent vertices $a$ and $b$. Assign color $1$ to $a$ and color $2$ to $b$. Define $\{a_3,\ldots,a_{d+1}\}=N(a)\setminus\{b\}$ and $\{b_3,\ldots,b_{d+1}\}=N(b)\setminus\{a\}$.
Because $G$ has no triangles, these vertices are distinct.
For the same reason, there are no edges among the $a_i$'s nor among the $b_i$'s.
There are also no edges between any $a_i$ and $b_j$, since such an edge would form a $4$-cycle together with $a$ and $b$.
The coloring is proper so far, with vertices $a$ and $b$ being CDVs.

For each $i\in\{3,\ldots,d+1\}$, let $A_i=N(a_i)\setminus\{a\}$ and $B_i=N(b_i)\setminus\{b\}$.
Each of the $A_i$'s and $B_i$'s contains $d-1$ vertices.
The $A_i$'s and $B_i$'s are pairwise disjoint, because a common vertex of two such sets would give a cycle of length at most $5$.
Moreover, there are no edges between vertices belonging to sets among the $A_i$'s and $B_i$'s, since such an edge would give a cycle of length at most $6$.
Neither is possible since the girth of $G$ is at least $7$.

We aim to construct a second CDV of color $2$. Choose a vertex $x\in A_3$ and assign color $2$ to it. Let $C=N(x)\setminus\{a_3\}$. The set $C$ contains $d-1$ uncolored vertices. To make $x$ a CDV, these vertices must receive the colors in $\{1,\ldots,d+1\}\setminus\{2,3\}$,
each exactly once.
We consider possible edges between $C$ and the sets defined above.
Since the girth of $G$ is at least $7$, no vertex of $C$ is adjacent to a vertex of $A_3\setminus\{x\}$ and there are no edges between $C$ and $A_i$ for any $i\in\{4,\ldots,d+1\}$.
Also, each vertex of $C$ has a neighbor in at most one of the sets $B_3,\ldots,B_{d+1}$, and no two vertices of $C$ have neighbors in the same set $B_i$. Otherwise, a cycle of length at most $6$ would be obtained. 
Assign the colors from $\{1,\ldots,d+1\}\setminus\{2,3\}$ to the vertices of $C$, using each color exactly once.
The resulting coloring is still proper, since the only colored vertex that is adjacent to one in $C$ is $x$.
Because $N[x]$ contains color $2$ on $x$, color $3$ on $a_3$, and every other color on the vertices of $C$, now $x$ is a CDV of color $2$.

We next construct a second CDV of color $1$. By the properties above, at most one vertex of $B_{d+1}$ has a neighbor in $C$. Since $|B_{d+1}|=d-1\geq 2$, we may choose a vertex $y\in B_{d+1}$ that has no neighbor in $C$. Assign color $1$ to $y$, and let $D=N(y)\setminus\{b_{d+1}\}$. The set $D$ contains $d-1$ uncolored vertices. To make $y$ a CDV, the vertices of $D$ must receive the colors in $\{2,3,\ldots,d\}$, each exactly once. Again, the girth condition gives several useful properties. 
No vertex of $D$ is adjacent to one in $B_{d+1}\setminus\{y\}$.
There are no edges between $D$ and $B_i$ for $i \le d$. Each vertex of $D$ has a neighbor in at most one of the sets $A_3,\ldots,A_{d+1}$, and no two vertices of $D$ have neighbors in the same set $A_i$. 
There is at most one edge between $C$ and $D$.
If such an edge exists, its endpoint in $D$ has no neighbor in $A_3$.
A violation of any of these properties would give a cycle of length at most $6$. 
We now color the vertices of $D$.
If there is an edge between $C$ and $D$, let $v$ be its endpoint in $D$ and assign color $2$ to $v$. This creates no conflict because $v$ is not adjacent to $x$ and no vertex of $C$ has color $2$.
If a vertex in $D$ has a neighbor in $A_3$, assign color 3 to that vertex.  This can also be done without conflict.
Assign the remaining unused colors from $\{2,\ldots,d\}$ to the remaining vertices of $D$, using each color exactly once.
The coloring is still proper since none of these vertices have neighbors that have already been colored.
Since $N[y]$ contains color $1$ on $y$, color $d+1$ on $b_{d+1}$, and every color from $2$ to $d$ on the vertices of $D$, we have made $y$ a CDV of color $1$.

We finish by coloring the remaining vertices in the $A_i$'s and $B_i$'s so that all of the $a_i$'s and $b_i$'s become CDVs.
For each $i\in\{3,\ldots,d+1\}$, we want the vertices in $A_i$ to get the colors in $\{1,\ldots,d+1\}\setminus\{1,i\}$, each exactly once, recalling that $x\in A_3$ already has color 2.
If $d=3$ and the remaining vertex in $A_3$ has a neighbor in $D$, then that neighbor has color 2 or 3, so we assign the vertex the required color 4 without conflict.
In any other case, at most one vertex in a given $A_i$ can have a neighbor (in $D$) that is already using one of the required colors.
Since such a vertex can have at most one such neighbor and there are at least two required colors, we assign a required color to such a vertex without conflict.
Similarly, we want the vertices in each $B_i$ to get the colors in $\{1,\ldots,d+1\}\setminus\{2,i\}$, each exactly once, recalling that $y\in B_{d+1}$ already has color 1.
If $d=3$ and the remaining vertex in $B_{d+1}$ has a neighbor in $C$, then that neighbor has color 1 or 4, so we assign the vertex the required color 3 without conflict.
In any other case, at most one vertex in a given $B_i$ can have a neighbor (in $C$) that is already using one of the required colors.
Since such a vertex can have at most one such neighbor and there are at least two required colors, we assign a required color to such a vertex without conflict.
For the remaining uncolored vertices among the $A_i$'s and $B_i$'s, there are no edges to vertices that already have a required color, so the remaining required colors may be assigned arbitrarily. Once these are complete, each $A_i$ (resp., $B_i$) contains the colors required to make $a_i$ (resp., $b_i$) a CDV. 

Now $a$, $b$, $a_3,\ldots,a_{d+1}$, $b_3,\ldots,b_{d+1}$, $x$, and $y$ are all CDVs. In particular, every color in $\{1,\ldots,d+1\}$ has at least two CDVs.
(Color $1$ has the CDVs $a$ and $y$, color $2$ has the CDVs $b$ and $x$, and for every $i\in\{3,\ldots,d+1\}$, the vertices $a_i$ and $b_i$ are CDVs of color $i$.) 
Finally, color all remaining vertices greedily. At every step, an uncolored vertex has at most $d$ colored neighbors, while $d+1$ colors are available. Therefore the coloring can be completed to a proper $(d+1)$-coloring of $G$. The previously constructed CDVs remain CDVs, and thus $G$ realizes $(2^{d+1})$.
\end{proof}

\medskip

Theorem~\ref{t:girth7} naturally raises the question whether the minimum girth assumption can be reduced from $7$ to $6$. As we have already argued, this is not possible for degree 2, since the cycle $C_7$ does not realize $(2,2,1)$. However, this example does not exclude the possibility for regular graphs of larger degree. The situation changes already for cubic graphs. The next theorem shows that every $3$-regular graph of girth at least $6$ realizes $(2^4)$.

\begin{theorem}\label{t:cubicGirth6}
Let $G$ be a $3$-regular graph of girth $g(G)\ge 6$.
Then $G$ realizes the sequence $(2,2,2,2)$.
\end{theorem}

\begin{proof}
If the girth of $G$ is at least $7$, then the result follows from Theorem~\ref{t:girth7}. We may therefore assume that the girth of $G$ is exactly $6$.
We will construct a proper 4-coloring of $G$ with at least two CDVs of each color.

Let $C=a_1a_2a_3a_4a_5a_6a_1$ be a $6$-cycle in $G$.
Since $G$ is $3$-regular, each vertex $a_i$ has exactly one neighbor outside $C$. Denote this neighbor by $b_i$. 
The vertices $b_1,\ldots,b_6$ are pairwise different, since if $b_i=b_j$ for some $i\neq j$, then this common neighbor together with a shortest path between $a_i$ and $a_j$ on $C$ would give a cycle of length at most $5$. 
Moreover, an edge between two vertices $b_i$ and $b_j$ is possible only when $a_i$ and $a_j$ are opposite vertices of the $6$-cycle, since otherwise the edge $b_ib_j$ together with a shortest path between $a_i$ and $a_j$ on $C$ would give a cycle of length at most $5$.
Thus the only possible edges among the $b_i$'s are $b_1b_4$, $b_2b_5$, and $b_3b_6$.
We consider two cases based on whether $b_1$ is adjacent to $b_4$.

\smallskip

First assume that $b_1b_4\in E(G)$. 
Let $c_1$ be the third neighbor of $b_1$ (distinct from $a_1$ and $b_4$) and let $c_4$ be the third neighbor of $b_4$ (distinct from $a_4$ and $b_1$).
The vertices $c_1$ and $c_4$ are different, since otherwise $b_1b_4c_1b_1$ would be a triangle. 
They are also nonadjacent, since an edge $c_1c_4$ would give the $4$-cycle $c_1b_1b_4c_4c_1$.
Assign colors $1,2,3,1,2,4$ to vertices $a_1,a_2,a_3,a_4,a_5,a_6$, respectively.
Assign colors $3,4,4,4,3,3$ to $b_1,b_2,b_3,b_4,b_5,b_6$, also respectively.
Assign color $2$ to both $c_1$ and $c_4$.
The coloring is proper so far.
Moreover, $a_1$ and $a_4$ are CDVs of color 1, $a_2$ and $a_5$ are CDVs of color 2, $a_3$ and $b_1$ are CDVs of color 3, and $a_6$ and $b_4$ are CDVs of color 4.

\smallskip

Next assume that $b_1b_4\not\in E(G)$.
Assign colors $3,2,1,4,2,1$ to vertices $a_1,a_2,a_3,a_4,a_5,a_6$, respectively.
Assign colors $4,4,3,3,3,4$ to $b_1,b_2,b_3,b_4,b_5,b_6$, also respectively.
The coloring is proper so far.
Moreover, $a_3$ and $a_6$ are CDVs of color 1, $a_2$ and $a_5$ are CDVs of color 2, $a_1$ is a CDV of color 3, and $a_4$ is a CDV of color 4.
For each $i\in\{1,4\}$, let $S_i=N(b_i)\setminus\{a_i\}$.
Since $G$ is $3$-regular, each set $S_i$ contains exactly two vertices.
Note that $S_1\neq S_4$, since otherwise there would be a 4-cycle.
So either they have a common vertex or they are disjoint.

Suppose that $S_1$ and $S_4$ both contain a vertex $z$.
Let $x$ and $y$ be the other vertices in $S_1$ and $S_4$, respectively.
Note that $x$ and $y$ cannot be adjacent, since otherwise $xyb_4zb_1x$ would be a 5-cycle.
Give $x$ and $y$ color 1 and give $z$ color 2.
The coloring is still proper.
Moreover, $b_4$ is a (second) CDV of color 3 and $b_1$ is a (second) CDV of color 4.

Now suppose that $S_1$ and $S_4$ are disjoint.
Let $\{x,y\}=S_1$ and $\{z,w\}=S_4$.
Note that a vertex in $S_1$ can be adjacent to at most one vertex in $S_4$, and vice versa, since otherwise we would get a 4-cycle.
Without loss of generality, we can assume $wx$ and $yz$ are not in $E(G)$.
Give $w$ and $x$ color 1 and give $y$ and $z$ color 2.
The coloring is still proper.
Moreover, $b_4$ is a (second) CDV of color 3 and $b_1$ is a (second) CDV of color 4.

\smallskip

Finally, color all remaining vertices greedily. At every step, an uncolored vertex has at most three colored neighbors, while four colors are available. Therefore the coloring can be completed to a proper $4$-coloring of $G$. The previously constructed CDVs remain CDVs. Thus $G$ realizes the sequence $(2,2,2,2)$.
\end{proof}

\medskip

Theorem~\ref{t:cubicGirth6} shows that the girth assumption in Theorem~\ref{t:girth7} can be reduced from $7$ to $6$ when $d=3$. It is natural to ask whether the same is true for every $d\geq 3$. There are some reasons to believe that this may be possible. The proof of Theorem~\ref{t:girth7} relies on the flexibility available when assigning the required colors within the second neighborhoods. As $d$ increases, these sets become larger, providing additional freedom to avoid conflicts. This leads to the following conjecture.

\begin{conjecture}\label{c:girth6}
Let $d\geq 3$ and let $G$ be a $d$-regular graph of girth $g(G)\ge 6$.
Then $G$ realizes the sequence $(2^{d+1})$.
\end{conjecture}

Our results %which results? add reference(s)?
show that girth $6$ and $7$ still allow local configurations which may obstruct sequence b-colorings corresponding to sequences with larger entries. We now return to the question raised at the beginning of the section. If the girth is at least $8$, the neighborhoods up to distance 3 from a chosen vertex do not contain the conflicts which appeared before. This allows us to extend Kouider's construction to the entire second neighborhood and realize the sequence $(d^d,1)$.

\begin{theorem}\label{t:girth8}
Let $d\geq 2$ and let $G$ be a $d$-regular graph of girth $g(G)\ge 8$.
Then $G$ realizes the sequence $(d^d,1)$.    
\end{theorem}

\begin{proof}
We construct a proper $(d+1)$-coloring of $G$ which realizes $(d^d,1)$.

Choose a vertex $x\in V(G)$ and write $N(x)=\{x_1,\ldots,x_d\}$. Assign color $d+1$ to $x$ and color $i$ to $x_i$ for each $i\in\{1,\ldots,d\}$. Thus $N[x]$ contains every color and $x$ is a CDV of color $d+1$. 

For each $i\in\{1,\ldots,d\}$, let $B_i=N(x_i)\setminus\{x\}$. Since $G$ is $d$-regular, each set $B_i$ contains exactly $d-1$ vertices. The sets $B_1,\ldots,B_d$ are pairwise disjoint. Indeed, if a vertex belonged to both $B_i$ and $B_j$, where $i\neq j$, then it would form a $4$-cycle together with $x_i$, $x$, and $x_j$. Moreover, each set $B_i$ is independent, since an edge between two vertices of some $B_i$ would form a triangle with $x_i$. There are also no edges between $B_i$ and $B_j$ for distinct $i$ and $j$, since such an edge would form a $5$-cycle through $x_i$, $x$, and $x_j$. For each $i\in\{1,\ldots,d\}$, color the vertices of $B_i$ so that each color in $\{1,\ldots,d\}\setminus\{i\}$ is used exactly once. This gives a proper coloring of all vertices at distance at most $2$ from $x$. For each $i\in\{1,\ldots,d\}$, $N[x_i]$ contains color $i$ on $x_i$, color $d+1$ on $x$, and every other color in $\{1,\ldots,d\}$ on the vertices of $B_i$. Hence each vertex $x_i$ is a CDV.

We next make all of the vertices in the $B_i$'s into CDVs. 
For each $i\in\{1,\ldots,d\}$ and each $j\in\{1,\ldots,d\}\setminus\{i\}$, let $u_{ij}$ be the vertex of $B_i$ having color $j$ and set $C_{ij}=N(u_{ij})\setminus\{x_i\}$.
The sets $C_{ij}$ are pairwise disjoint:
a common vertex of $C_{ij}$ and $C_{ij'}$ with $j\neq j'$ would form a $4$-cycle with $u_{ij}$, $x_i$, and $u_{ij'}$;
a common vertex of $C_{ij}$ and $C_{i'j'}$ with $i\neq i'$ would form a 6-cycle through $u_{ij}$, $x_i$, $x$, $x_{i'}$ and $u_{i'j'}$.
There are no edges within or among the $C_{ij}'s$:
an edge between two vertices of a single $C_{ij}$ would form a triangle with $u_{ij}$;
an edge between vertices of $C_{ij}$ and $C_{ij'}$ with $j\neq j'$ would form a 5-cycle with $u_{ij}$, $x_i$ and $u_{ij'}$;
an edge between vertices of $C_{ij}$ and $C_{i'j'}$ with $i\neq i'$ would form a 7-cycle with $u_{ij}$, $x_i$, $x$, $x_{i'}$ and $u_{i'j'}$.
All these possibilities are excluded because the girth is at least $8$.
Color the vertices of each $C_{ij}$ with the colors in $\{1,\ldots,d+1\}\setminus\{i,j\}$, using each color exactly once.
This is possible because, due to $G$ being $d$-regular, there are exactly $d-1$ vertices in $C_{ij}$.
The coloring remains proper due to our observations about edge absences.
Now $N[u_{ij}]$ contains color $j$ on $u_{ij}$ itself, color $i$ on $x_i$, and every color in $\{1,\ldots,d+1\}\setminus\{i,j\}$ on the vertices of $C_{ij}$.
Hence each vertex in each set $B_i$ is a CDV.

We now count the CDVs of each color. Fix a color $j\in\{1,\ldots,d\}$. The vertex $x_j$ is a CDV of color $j$. Moreover, color $j$ appears exactly once in every set $B_i$ with $i\neq j$. Since there are $d-1$ such sets, color $j$ has exactly $d-1$ additional CDVs in the second neighborhood of $x$. Therefore color $j$ has at least $d$ CDVs in total. The vertex $x$ is a CDV of color $d+1$.

Finally, color all remaining vertices greedily. At every step, an uncolored vertex has at most $d$ colored neighbors, while $d+1$ colors are available. Therefore the coloring can be completed to a proper $(d+1)$-coloring of $G$, and the previously constructed CDVs remain CDVs. Thus each of the colors $1,\ldots,d$ contains at least $d$ CDVs, and color $d+1$ contains at least one CDV. Hence $G$ realizes the sequence $(d^d,1)$.
\end{proof}

%-----------------
\section{One additional color-dominating vertex}\label{s:oneAdd}
%-----------------

The sequence $(1^\infty)$ corresponds exactly to the classical notion of a b-coloring. Indeed, every graph admits a b-coloring with $\varphi(G)$ colors and therefore realizes the sequence $(1^{\varphi(G)})$. The next natural question is what happens if we require only one additional color-dominating vertex. This leads to the sequence $(2,1^\infty)$. According to Definition~\ref{d:sequence}, we ask whether a graph realizes a sequence of the form $(2,1^k)$ for some positive integer $k$. It is then natural to ask which graphs realize this sequence and, more importantly, whether the graphs that do not realize it can be characterized.

The complete graphs $K_1$ and $K_2$ clearly do not realize $(2,1^\infty)$. The graph $K_1$ has only one vertex, while the unique b-coloring of $K_2$ contains exactly one CDV of each color. As these are the only connected graphs on fewer than three vertices, we assume throughout this section that $|V(G)|\ge 3$.

Some positive results are already known from the previous sections. In particular, Theorem~\ref{t:girth6} shows that every $d$-regular graph with $d\ge 2$ and girth at least $6$ realizes $(2,1^d)$, and therefore also realizes $(2,1^\infty)$.

Another important class is formed by connected bipartite graphs. Every such graph admits a proper coloring with two colors. Since it has at least three vertices, one of the two color classes contains at least two vertices. Moreover, every vertex is a CDV, so this coloring realizes the sequence $(2,1)$. Therefore, it remains to study non-bipartite graphs $G$, i.e., the graphs with $\chi(G)\geq3$. We first consider graphs with chromatic number $3$.

For nonnegative integers $r$, $s$, and $t$, let $T_{r,s,t}$ denote the graph obtained from a triangle (the complete graph $K_3$) by attaching $r$, $s$, and $t$ leaves to its three vertices, respectively. Every graph $T_{r,s,t}$ has $\chi(T_{r,s,t})=\varphi(T_{r,s,t})=3$. Moreover, no graph from this family realizes the sequence $(2,1,1)$, since the only vertices that can be CDVs are the three vertices of the triangle.

Another example is the cycle $C_5$, which also has $\chi(C_5)=\varphi(C_5)=3$. Up to a permutation of the colors, its only proper coloring with three colors is obtained by assigning the colors $1,2,1,2,3$ consecutively around the cycle. In this coloring, there is exactly one CDV of each color. Therefore, $C_5$ does not realize the sequence $(2,1,1)$.

Since both $T_{r,s,t}$ and $C_5$ have chromatic number and b-chromatic number equal to $3$, the only possible candidate is the sequence $(2,1,1)$. As neither graph realizes this sequence, neither realizes $(2,1^\infty)$. The next theorem shows that these are the only such graphs.

\begin{theorem}\label{t:chi3}
Let $G$ be a connected graph with $\chi(G)=3$. 
Then $G$ realizes the sequence $(2,1^\infty)$ if and only if $G$ is neither $C_5$ nor $T_{r,s,t}$ for any nonnegative integers $r$, $s$, and $t$. 
\end{theorem}

\begin{proof}
We first prove the forward implication by contraposition. Suppose that $G$ is either $C_5$ or a graph $T_{r,s,t}$ for some nonnegative integers $r$, $s$, and $t$. As shown before the theorem, both $C_5$ and every graph $T_{r,s,t}$ have chromatic number and b-chromatic number equal to $3$, but do not realize the sequence $(2,1,1)$. Since this is the only possible sequence of the form $(2,1^k)$ that they could realize, neither $C_5$ nor any graph $T_{r,s,t}$ realizes $(2,1^\infty)$.

For the converse, suppose that $G$ is a connected graph with $\chi(G)=3$ which is neither $C_5$ nor any graph $T_{r,s,t}$. We show that $G$ realizes $(2,1,1)$, and hence also $(2,1^\infty)$.

Since $\chi(G)=3$, there exists a proper coloring of $G$ using three colors.
As discussed in Section~\ref{s:defs}, this must be a b-coloring.
Consider such a coloring.
If one of the three color classes already contains at least two CDVs, then this coloring realizes $(2,1,1)$, and hence also $(2,1^\infty)$. Therefore, we may assume that each color class contains exactly one CDV. Let $x$, $y$, and $z$ denote the CDVs of colors $1$, $2$, and $3$, respectively. The subgraph induced by these three vertices is one of $K_3$, $P_3$, $\overline{P_3}$, or $\overline{K_3}$. We consider these four cases separately.

\smallskip

Suppose first that the vertices $x$, $y$, and $z$ induce $K_3$. If every neighbor of $x$, $y$, and $z$ outside the induced $K_3$ is a leaf, then, since $G$ is connected, $G$ is obtained from the induced $K_3$ by attaching leaves to its vertices. Hence $G$ is $T_{r,s,t}$ for some nonnegative integers $r$, $s$, and $t$, contrary to our assumption. 
Therefore, at least one of the vertices $x$, $y$, and $z$ has a non-leaf neighbor outside the induced $K_3$. 
Permuting colors if necessary,
let $u$ be a non-leaf neighbor of $z$ such that $u$ has color $1$. 
Since $u$ is not a leaf, it has a neighbor $v\neq z$.
As $u$ is not a CDV, it has no neighbor of color $2$, so
$v$ must
have color $3$.
This means that $v\not\in\{x,y\}$.
Since $v$ is not the 
unique
CDV of color $3$, it has no neighbor of color $2$.
We can thus recolor $v$ with color $2$ while preserving a proper coloring.
This keeps $x$, $y$ and $z$ as CDVs and makes $u$ a second CDV of color $1$.
The modified coloring thus realizes $(2,1,1)$ and hence also $(2,1^\infty)$.

\smallskip
Suppose next that $x$, $y$ and $z$ induce $P_3$.
Permuting colors if necessary, we may assume that $y$ is adjacent to $x$ and $z$.
(Recall that $x$, $y$ and $z$ have colors 1, 2 and 3, respectively.)
Since $x$ is a CDV, it is adjacent to a vertex $u$ that has color 3.
Similarly, $z$ is adjacent to a vertex $v$ that has color 1.
Let $A$ denote $\{u,x,y,z,v\}$ and set $A'=V(G)\setminus A$.
We proceed based on whether or not $u$ is adjacent to $v$.

Suppose that $u$ is adjacent to $v$, so $A$ induces a 5-cycle.
Since $G$ is assumed not to be $C_5$ itself, $A'$ is nonempty.
We distinguish three cases:
a1) at least one vertex in $A$ has a neighbor in $A'$ that is adjacent to another vertex in $A'$;
a2) no vertex in $A'$ has neighbors in both $A$ and $A'$ but at least one pair of vertices in $A$ have a common neighbor in $A'$;
a3) every vertex in $A'$ is a leaf.
In each case, we will show that $G$ can be recolored to realize $(2,1,1)$.

a1)
Suppose vertex $w_1\in A$ is adjacent to $w_2\in A'$, which is adjacent to $w_3\in A'$.
Permuting colors if necessary, we may assume that $w_1$ has color 1 and $w_2$ has color 2.
Since $w_2$ is not a CDV, $w_3$ must have color 1.
Since $w_3$ is not a CDV, it cannot have any neighbors of color 3.
We can thus recolor $w_3$ with color 3 while preserving properness of the coloring or changing the CDV status of $x$, $y$ or $z$.
Moreover, this recoloring turns $w_2$ into an extra CDV, as required.

a2)
Such a pair of vertices must have the same color, since otherwise the common neighbor would be a CDV.
Swapping colors 1 and 3 if necessary, we may assume that $x$ and $v$ are both adjacent to $w\in A'$.
We will recolor $G$.
Give vertices $u$, $x$, $y$, $z$, $v$ and $w$ the colors 1, 2, 3, 1, 3 and 1, respectively.
Any remaining vertices can be colored greedily since they are either leaves or perhaps adjacent only to $u$ and $z$.
Now we have a proper coloring where $u$ and $w$ are CDVs of color 1 and where $x$ and $y$ are CDVs of colors 2 and 3, respectively.
This new coloring thus realizes $(2,1,1)$.

a3)
Suppose $w\in A$ is adjacent to the leaf $w'\in A'$.
We will recolor $G$.
Starting at $w$ and going around the cycle, use the colors $1,2,3,1,2$.
This produces three CDVs in $A\setminus\{w\}$, one of each color.
Now give $w'$ color 3.
Any remaining vertices can be colored without causing conflicts.
Moreover, $w$ is a second CDV of color 1, as required.

Now suppose that $u$ is not adjacent to $v$, so $A$ induces a path.
We again distinguish three cases:
b1) at least one of $u$ or $v$ has a neighbor in $A'$;
b2) vertices $u$ and $v$ are leaves, but another vertex in $A$ has a non-leaf neighbor in $A'$;
b3) no vertex in $A$ has a non-leaf neighbor in $A'$.
We show that in cases b1 and b2, $G$ can be recolored to realize $(2,1,1)$ and that case b3 does not occur.

b1)
Suppose $u$ is adjacent to $w\in A'$.
Since $u$ is not a CDV, $w$ has color 1.
Because $w$ is not a CDV, it cannot have any neighbors of color 2.
We can thus recolor $w$ with color 2 while preserving the propriety of the coloring.
This turns $u$ into an extra CDV, as required.
The case in which $v$ has a neighbor in $A'$ is symmetric.

b2)
Let $w_1\in\{x,y,z\}$ be adjacent to $w_2\in A'$, which is itself adjacent to $w_3\neq w_1$.
Permuting colors as necessary, let $w_1$ have color 1 and $w_2$ have color 2.
Since $w_2$ is not a CDV, $w_3$ must have color 1.
This means that $w_3\in A'$, since $u$ and $v$ are leaves and $w_1$ is the only other vertex in $A$ that has color 1.
Because $w_3$ is not a CDV, it cannot have any neighbors of color 3.
Recolor $w_3$ with color 3.
This preserves the propriety of the coloring.
It also keeps $x$, $y$ and $z$ as CDVs.
Moreover, it turns $w_2$ into an extra CDV, as required.

b3)
In this case, $G$ is bipartite, contradicting the assumption $\chi(G)=3$.

\smallskip

Suppose now that $x$, $y$, and $z$ induce $\overline{P_3}$.
Recall that these vertices are so named so that they have colors 1, 2 and 3, respectively.
Permuting colors if necessary, we may assume that $y$ and $z$ are adjacent, while $x$ is adjacent to neither of them.
Let $X_2$ and $X_3$ denote the sets of neighbors of $x$ with colors $2$ and $3$, respectively. Since $x$ is a CDV, both sets are nonempty. Let $Y_1$ denote the set of neighbors of $y$ with color $1$, and let $Z_1$ denote the set of neighbors of $z$ with color $1$. These sets are also nonempty because $y$ and $z$ are CDVs. 

We may assume that every neighbor of $y$ different from $z$ has color $1$ and is thus in $Y_1$.
Indeed, suppose that $v\neq z$ is a neighbor of $y$ with color $3$.
Since $v$ is not a CDV and is adjacent to $y$ of color $2$, $v$ has no neighbor of color $1$.
We may thus give $v$ color 1 without disturbing the propriety of the coloring or the CDV status of $x$, $y$ and $z$.
By symmetry, we may also assume that every neighbor of $z$ different from $y$ belongs to $Z_1$.

We may also assume that we have accounted for all of the vertices, so $V(G)=\{x,y,z\}\cup X_2\cup X_3\cup Y_1\cup Z_1$.
For example, suppose that a vertex $v\in X_2$ has a neighbor $u\notin \{x,y,z\}\cup X_2\cup X_3\cup Y_1\cup Z_1$. 
Since $v$ (with color 2) is not a CDV and is adjacent to $x$ (of color $1$), $u$ must have color $1$.
Moreover, $u$ is not a CDV and is adjacent to $v$ of color $2$, so it has no neighbor of color $3$. We may therefore recolor $u$ with color $3$, making $v$ a second CDV of color $2$. 
This modified coloring realizes $(2,1,1)$, and hence also $(2,1^\infty)$.
By symmetry, the same argument applies to every vertex in $X_3$, $Y_1$, and $Z_1$.

Let us consider the edges among the four sets $X_2$, $X_3$, $Y_1$ and $Z_1$. 
There are no edges inside any one of them, since each set is monochromatic.
There are also no edges between $Y_1$ and $Z_1$, since all vertices in these two sets have color $1$. 
An edge between $X_2$ and $X_3$ would make both endpoints CDVs.
An edge between $X_2$ and $Z_1$ would make its endpoint in $Z_1$ a second CDV of color $1$, while an edge between $X_3$ and $Y_1$ would make its endpoint in $Y_1$ a second CDV of color $1$.
Therefore, the only possible edges among these sets are between $X_2$ and $Y_1$ and between $X_3$ and $Z_1$.
There must be at least one edge of the latter type.
Indeed, if no vertex in $X_3$ were adjacent to any vertex in $Z_1$, then we could properly recolor the vertices so those in $X_3\cup X_2\cup\{y\}\cup Z_1$ have color 1 and the rest have color 2.
But this would show that $G$ is bipartite, contradicting the assumption $\chi(G)\geq3$.

We now recolor some vertices.
Assign color 2 to every vertex in $Z_1$, color 1 to $z$ and color 3 to $y$.
By the above description of the possible edges of $G$, this coloring is proper. 
Choose an edge $uv$ with $u\in X_3$ and $v\in Z_1$.
Then $u$ is a CDV of color $3$, since it is adjacent to $x$ of color $1$ and to $v$ of color $2$, while $v$ is a CDV of color $2$, since it is adjacent to $z$ of color $1$ and to $u$ of color $3$. 
We have not recolored the neighbors of $x$, so it is still a CDV of color 1.
However, $z$ is now adjacent to $y$ of color 3 and $v$ of color 2, so it is a second CDV of color 1.
This new coloring realizes $(2,1,1)$, and hence also $(2,1^\infty)$.

\smallskip

Suppose finally that the vertices $x$, $y$, and $z$ induce $\overline{K_3}$. Recall that $x$, $y$, and $z$ have colors $1$, $2$, and $3$, respectively. Let $X_2$ and $X_3$ denote the sets of neighbors of $x$ with colors $2$ and $3$, respectively. Similarly, let $Y_1$ and $Y_3$ denote the sets of neighbors of $y$ with colors $1$ and $3$, and let $Z_1$ and $Z_2$ denote the sets of neighbors of $z$ with colors $1$ and $2$. Since $x$, $y$, and $z$ are CDVs, all six sets are nonempty. 
These sets are pairwise disjoint, since any vertex in two of them would be an extra CDV.

As in the previous case, we consider whether there are any more vertices.
Suppose, for example, that a vertex $v\in X_2$ has a neighbor
$u\notin\{x,y,z\}\cup X_2\cup X_3\cup Y_1\cup Y_3\cup Z_1\cup Z_2$.
Since $v$ is not a CDV and is adjacent to $x$ of color $1$, it has no neighbor of color $3$. Thus $u$ has color $1$. Moreover, $u$ is not a CDV and is adjacent to $v$ of color $2$, so it has no neighbor of color $3$. We may therefore recolor $u$ with color $3$, making $v$ a second CDV of color $2$. This modified coloring realizes $(2,1,1)$, and hence also $(2,1^\infty)$. The same argument applies to every vertex in the other five sets. Therefore, we may assume that
$V(G)=\{x,y,z\}\cup X_2\cup X_3\cup Y_1\cup Y_3\cup Z_1\cup Z_2$.

There are clearly no edges between vertices of the same color. 
An edge between $X_2$ and $Y_3$ would make its endpoint in $X_2$ a CDV of color $2$.
Analogous statements hold for most of the other pairs.
In the end, we find that the only possible edges among the six sets are between $X_2$ and $Y_1$, between $X_3$ and $Z_1$, and between $Y_3$ and $Z_2$. 
Since $G$ is connected, at least two of these three types of edge must be present.
By symmetry, assume that there is an edge between $X_2$ and $Y_1$ and an edge between $X_3$ and $Z_1$.

We now define a new coloring. Assign color $1$ to $x$, $z$, and every vertex of $Y_1\cup Y_3$. Assign color $2$ to every vertex of $X_2\cup Z_1\cup Z_2$. Finally, assign color $3$ to $y$ and every vertex of $X_3$. By the above description of the possible edges, this coloring is proper. Choose an edge $uv$ with $u\in X_2$ and $v\in Y_1$. Then $v$ is a CDV of color $1$, since it is adjacent to $u$ of color $2$ and to $y$ of color $3$. The vertex $x$ is also a CDV of color $1$, since it has neighbors in both $X_2$ and $X_3$. Similarly, choose an edge $wp$ with $w\in X_3$ and $p\in Z_1$. Then $w$ is a CDV of color $3$, since it is adjacent to $x$ of color $1$ and to $p$ of color $2$. The vertex $p$ is a CDV of color $2$, since it is adjacent to $z$ of color $1$ and to $w$ of color $3$. This new coloring realizes $(2,1,1)$, and hence also $(2,1^\infty)$.
\end{proof}

\medskip

Theorem~\ref{t:chi3} completely characterizes graphs with chromatic number $3$ that do not realize the sequence $(2,1^\infty)$. For graphs with larger chromatic number, such a characterization is currently unknown.
We therefore turn our attention to simple structural conditions that guarantee that a graph does not realize $(2,1^\infty)$. The following proposition provides one such condition.

\begin{proposition}\label{p:no21}
Suppose that $G$ contains a clique of size $k$ and every vertex outside this clique has degree strictly less than $k-1$. Then $G$ does not realize the sequence $(2,1^\infty)$.
\end{proposition}

\begin{proof}
Let $K$ be the clique of size $k$. Since the vertices of $K$ are pairwise adjacent, every proper coloring of $G$ uses at least $k$ colors.

Suppose that $G$ admits a b-coloring with $q \ge k$ colors. Every CDV in this coloring has degree at least $q-1$. Since every vertex outside $K$ has degree strictly less than $k-1$, no vertex outside $K$ can be a CDV. Therefore, all CDVs belong to $K$. As every color class must contain a CDV and $K$ contains only $k$ vertices, we obtain $q\le k$. Hence $q=k$.

The vertices of $K$ receive pairwise different colors. Moreover, each of them is a CDV, since its closed neighborhood contains all $k$ colors used on $K$. On the other hand, no vertex outside $K$ can be a CDV. Thus every color class contains exactly one CDV, and consequently $G$ does not realize $(2,1^\infty)$.
\end{proof}

\medskip

The converse of Proposition~\ref{p:no21} does not hold. In particular, a graph may contain a clique of size $k$ and vertices outside this clique of degree (at least) $k-1$, while still not realizing $(2,1^\infty)$. The graph in Figure~\ref{f:clique} provides such an example. It contains a clique of size $4$, while each of the two remaining vertices has degree $3$ yet the graph can easily be shown not to realize $(2,1^\infty)$.

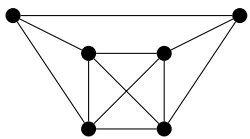
\begin{figure}[ht]
\centering
\begin{tikzpicture}
\tikzset{V/.style={shape=circle,fill=black,inner sep=2pt,outer sep=0pt}}
\node[V] (a) at (1,0) {};
\node[V] (b) at (2,0) {};
\node[V] (c) at (1,1) {};
\node[V] (d) at (2,1) {};
\draw (a)--(b)--(c)--(d)--(a) (a)--(c) (b)--(d);
\node[V] (e) at (0,1.5) {};
\node[V] (f) at (3,1.5) {};
\draw (a)--(e)--(c) (b)--(f)--(d) (e)--(f);
\end{tikzpicture}
\caption{A counterexample to the converse of Proposition~\ref{p:no21}}
\label{f:clique}
\end{figure}

%-----------------
\section{Concluding remarks and open questions}
%-----------------

In this paper, we introduced sequence b-colorings as a generalization of classical b-colorings in which the required number of CDVs may vary among color classes. We established some fundamental properties of the associated parameters, proved that every sequence is realizable, and showed that the corresponding decision problem is NP-complete. We also investigated sequence b-colorings of regular graphs, including a complete characterization for cycles, and studied colorings requiring only one additional CDV. The results obtained in these directions leave several open questions.

Theorem~\ref{t:girth7} shows that every $d$-regular graph with $d\ge 3$ and girth at least $7$ realizes $(2^{d+1})$. Theorem~\ref{t:cubicGirth6} improves the girth assumption from $7$ to $6$ for cubic graphs. This suggests that the same improvement may hold for every degree at least $3$, as stated in Conjecture~\ref{c:girth6}.

\begin{question}
Does every $d$-regular graph with $d\ge 3$ and girth at least $6$ realize the sequence $(2^{d+1})$?
\end{question}

An affirmative answer would show that the obstruction given by the cycle $C_7$ is specific to the case $d=2$. It would also extend Theorem~\ref{t:cubicGirth6} from cubic graphs to regular graphs of arbitrary degree. More generally, the results of the section on regular graphs determine sufficient girth conditions for several natural sequences. This raises the broader problem of understanding the relationship between the entries in a sequence and the girth required to guarantee its realization.

\begin{question}
Let $d\ge 2$, and let $S$ be a sequence of length $d+1$. What is the minimum integer $g$ such that every $d$-regular graph of girth at least $g$ realizes $S$?
\end{question}

Section~\ref{s:oneAdd} considered the sequence $(2,1^\infty)$, which represents the smallest possible strengthening of the classical b-coloring requirement. Theorem~\ref{t:chi3} gives a complete characterization of connected graphs with chromatic number $3$ that do not realize this sequence. For graphs with larger chromatic number, however, no corresponding characterization is currently known.

\begin{question}
For each integer $k\geq4$, which connected graphs $G$ with $\chi(G)=k$ do not realize the sequence $(2,1^\infty)$?
\end{question}

Finally, most of the results in this paper concern whether a graph $G$ realizes a prescribed sequence. The associated extremal parameters and spectra have received much less attention. For a fixed sequence $S$, the $S$-spectrum records all numbers of colors with which $S$ can be realized, while $\chi_S(G)$ and $\varphi_S(G)$ record its minimum and maximum elements. The results on regular graphs also determine $\varphi_S(G)$ in several cases. Indeed, whenever a $d$-regular graph realizes a sequence of length $d+1$, the general upper bound $\varphi_S(G)\leq\Delta(G)+1=d+1$ implies that $\varphi_S(G)=d+1$. Outside this setting, however, the behaviors of $\chi_S(G)$, $\varphi_S(G)$, and the intermediate elements of the $S$-spectrum remain largely unexplored.

\begin{question} For a fixed sequence $S$, which sets of positive integers can occur as the $S$-spectrum of a graph?
\end{question}

This question contains several related problems. One may study the possible values of $\chi_S(G)$ and $\varphi_S(G)$, find bounds in terms of other graph parameters, or determine these parameters exactly for particular graph classes. It would also be interesting to understand when the $m_S$-degree gives the exact value of $\varphi_S(G)$.

%-----------------
\section*{Acknowledgements}
%-----------------

M.\ Jakovac was supported by the Slovenian Research and Innovation Agency (ARIS) under the grants  P1-0297, N1-0285, N1-0431.

%-----------------
\section*{Declaration of interests}
%-----------------
 
The authors declare that they have no conflict of interest. 

%-----------------
\section*{Data availability}
%-----------------
 
Our manuscript has no associated data.

\end{document}